\documentclass[twoside,12pt,a4paper]{amsart}
\usepackage{amssymb}
\usepackage{mathrsfs}
\usepackage{amsmath}
 \usepackage[colorlinks,citecolor=blue,urlcolor=blue, linkcolor=blue, backref]{hyperref}
 \usepackage{undertilde}
 \usepackage[capitalise]{cleveref}		
 \usepackage{upgreek}
 \usepackage{tikz}
 
\newtheorem{theorem}{Theorem}[section]
\newtheorem*{theorem*}{Theorem}

\newtheorem{corollary}[theorem]{Corollary}
\newtheorem{definition}[theorem]{Definition}

\newtheorem{fact}[theorem]{Fact}
\newtheorem{lemma}[theorem]{Lemma}

\newtheorem{proposition}[theorem]{Proposition}
\newtheorem{question}[theorem]{Question}

\def\uh{\upharpoonright}

\newcommand{\CK}{\omega_1^{\mathrm{CK}}}
\newcommand{\KO}{\mathcal{O}}

\newcommand{\AC}{\mathrm{AC}}

\newcommand{\bi}{\begin{itemize}}
\newcommand{\ei}{\end{itemize}}
\newcommand{\bc}{\begin{center}}
\newcommand{\ec}{\end{center}}

\def\uh{\upharpoonright}

\newcommand{\D}{\mathcal{D}}
\newcommand{\R}{\mathbb{R}}

\newcommand{\wDC}{\mathrm{wDC}_{\mathbb{R}}}
 
\renewcommand{\DH}{\mathrm{Dim_H}}
\newcommand{\DP}{\mathrm{Dim_P}}

\newcommand{\ZF}{{\mathrm{ZF}}}
\newcommand{\TD}{{\mathrm{TD}}}
\newcommand{\sTD}{{\mathrm{sTD}}}
\newcommand{\AD}{{\mathrm{AD}}}
 \newcommand{\DCR}{{\mathrm{DC}_{\mathbb{R}}}}
\newcommand{\CCR}{{\mathrm{CC}_{\mathbb{R}}}}
\newcommand{\x}{{\mathbf{x}}}
\newcommand{\y}{{\mathbf{y}}}
\newcommand{\z}{{\mathbf{z}}}

\begin{document}
\title{Some consequences of $\TD$ and $\sTD$}
 \author{Yinhe Peng, Liuzhen Wu and Liang Yu }
 \thanks{We thank  Denis Hirschfeldt, Ramez  Sami,  Theodore  Slaman, and Hugh Woodin for their helpful discussions. Especially for Slaman's introducing Hausdorff dimension regularity to us.  Peng was partially supported by NSF of China No. 11901562 and the Hundred Talents Program of the Chinese Academy of Sciences. Wu was partially supported by NSFC No. 11871464. Yu was partially supported by NSF of China No. 11671196 and 12025103.}
 \address{Institute of Mathematics, Chinese Academy of Sciences\\ 
 East Zhong Guan Cun Road No. 55\\Beijing 100190\\China}
 \email{pengyinhe@amss.ac.cn}
  \address{HLM\\Academy of Mathematics and Systems Science\\Chinese Academy of Sciences\\East Zhong Guan Cun Road No. 55\\ Beijing 100190\\China }
 \email{lzwu@math.ac.cn}
\address{Department of Mathematics\ \\
Nanjing University, Jiangsu Province 210093\\
P. R. of China} \email{yuliang.nju@gmail.com.}
 
\maketitle
\begin{abstract}
 Strongly Turing determinacy, or $\sTD$,  says that for any set $A$ of reals, if $\forall x\exists y\geq_T x (y\in A)$, then there is a pointed set $P\subseteq A$. We prove the following consequences of Turing determinacy ($\TD$) and  $\sTD$:
\begin{enumerate}
\item[(1)]$\ZF+\TD$ implies weakly dependent choice ($\wDC$).
\item[(2)]$\ZF+\sTD$ implies that every set of reals is measurable and has Baire property.
\item[(3)] $\ZF+\sTD$ implies that every uncountable set of reals has a perfect subset.
\item[(4)]$\ZF+\sTD$ implies that for any set of reals $A$ and any $\epsilon>0$, 
\begin{itemize}
\item[(a)] there is a closed set $F\subseteq A$ so that $\DH(F)\geq \DH(A)-\epsilon$.
\item[(b)]there is a closed set $F\subseteq A$ so that $\DP(F)\geq \DP(A)-\epsilon$.
\end{itemize}
\end{enumerate}
 \end{abstract}
\section{Introduction} \label{sec:introduction}
\subsection{$\TD$ and $\sTD$}
Turing reduction $\leq_T$ is a partial order over reals. It naturally induces an equivalence relation $\equiv_T$. Given a real $x$, its corresponded  Turing degree $\x$ is a set of reals defined as $\{y\mid y\equiv_T x\}$. We say $\x\leq \mathbf{y}$ if $x\leq_T y$.  We use $\D$ to denote the set of Turing degrees. An {\em upper cone}   of Turing degrees is the set $\{\mathbf{y}\mid \mathbf{y}\geq \mathbf{x}  \}$.

We say that a perfect set $P$ is {\em pointed} if there is a perfect tree $T\subseteq 2^{<\omega}$ so that $[T]=P$ and for any $x\in P$, $T\leq_T x$, where $[T]=\{x\in 2^{\omega}\mid \forall n (x\uh n \in T)\}$.
\begin{definition}
\begin{itemize}
\item  Turing determinacy,  or $\TD$, says that for any set $A$ of Turing degrees, either $A$ or $\D\setminus A$ contains an upper cone of Turing degrees.
\item  Strongly Turing determinacy, or $\sTD$,  says that for any set $A$ of reals, if $\forall x\exists y\geq_T x (y\in A)$, then there is a pointed set $P$ so that $P\subseteq A$.
\end{itemize}
\end{definition}

Martin proves the following famous theorem.
\begin{theorem}[Martin \cite{Martin68}] 
Over $\ZF$, Axiom of determinacy, or $\AD$,  implies $\TD$.
\end{theorem}

\begin{definition}
\begin{itemize}
 \item Countable choice axiom for sets of reals,  or $\CCR$, says that for any countable sequence $\{A_n\}_{n\in \omega}$ of nonempty sets of reals, there is a function $f:\omega\to \mathbb{R}$ so that for every $n$, $f(n)\in A_n$.
 \item  Dependent choice axiom for sets of reals,  or $\DCR$, says that for any binary relation $R$ over reals so that $\forall x \exists y R(x,y)$, there is a function $f:\omega\to \mathbb{R}$ so that for every $n$, $R(f(n), f(n+1))$.
 \end{itemize}
\end{definition}

Turing determinacy is an important  and very useful consequence of $\AD$. First it is a neat statement, which looks more like an axiom than $\AD$. Second, in many situations, $\TD$ seems sufficient to be used to prove set theory theorems.  The following theorem justifies this phenomenon. 
\begin{theorem}[Woodin]\label{theorem:ad vs td in lr}
Assume $\ZF+V=L(\mathbb{R})+\DCR$. $\AD$ is equivalent to $\TD$.
\end{theorem}


Moreover, as we shall see in this paper, different than $\AD$ which often needs a very genius, tricky and case-by-case design of games, $\TD$ (and $\sTD$) usually provides a more consistent way to solve problems in set theory.

The first result in this paper concerns the relationship between $\AD$ and Axiom of Choice, or $\AC$.

Though $\AD$ contradicts to $\AC$,  Mycielski proves the following theorem.
\begin{theorem}[Mycielski \cite{MYC64}]
Over $\ZF$,  $\AD$ implies $\CCR$.
\end{theorem}

The question if $\AD$ implies $\DCR$ remains a long time. 
\begin{question}[Solovay]\label{querstion: ad vs dc}
Over $\ZF$, does $\AD$ imply $\DCR$?
\end{question}

Kechris proves the following result.
\begin{theorem}[Kechris \cite{Kechris84}]
Assuem $\ZF+V=L(\mathbb{R})$, $\AD$ implies $\DCR$.
\end{theorem}

It is unknown whether the assumption $V=L(\mathbb{R})$ can be removed. Recently, the following ``unconditional" result is proved.

\begin{theorem}[Peng and Yu \cite{PY21}]\label{theorem: td ccr}
Over $\ZF$,   $\TD$ implies $\CCR$.
\end{theorem}
We will use $\CCR$ throughout the paper even without mentioning it. 

\smallskip

The first result in this paper is  a partial solution to Question \ref{querstion: ad vs dc}. We  prove that $\ZF+\TD$ implies $\wDC$,  a weaker version of $\DCR$ (for the definition of $\wDC$, see Definition \ref{definition: wdc}).

The second result in this paper is about the regularity properties of sets of reals. Though $\TD$ seems unlikely as strong as $\AD$, a natural question is whether $\TD$ is as ``useful" as $\AD$.  Sami initiated  this project by proving (in \cite{Sami89}) that $\ZF+\TD$ implies $\mathrm{CH}$, the continuum hypothesis. But it seems a rather difficult (and long standing) question whether $\ZF+\TD(+\mathrm{DC})$ implies regularity properties for sets of reals.  In this paper, we give a partial answer to this question by showing that strongly Turing determinacy, $\sTD$--a stronger version of $\TD$,    implies the regularity properties for sets of reals.

A basis for a class $\mathscr{C}$ of linearly ordered sets is a collection $\mathscr{B}\subseteq \mathscr{C}$ such that for each $L_1\in \mathscr{C}$, there is an $L_2\in \mathscr{B}$ such that $L_2$ is isomorphic to a subset of $L_1$. Investigating basis for linear ordering   is a very active area in  set theory today. For example, Moore \cite{Moore06} proves that   under proper forcing axioms,  $\mathrm{PFA}$, a five-element basis exists. But it seems that   basis theorems for linear orderings under $\AD$  remains untouched. In this paper, we prove a basis theorem for linear orderings over $\mathbb{R}$ under the assumption $\ZF+\TD+\DCR+\mbox{``every uncountable set of reals has a perfect subset"}$ by showing that  for any linear oder $\leq_L$ over $\mathbb{R}$, there is an order preserving embedding from $(2^{\omega},\leq)$ to $(\mathbb{R},\leq_L)$. In other words, $\{(2^{\omega},\leq)\}$ is a basis for $\{(\mathbb{R},\leq_L)\mid \leq_L \mbox{ is a linear ordering over } \mathbb{R}\}.$

The last result in this paper is an application of recursion theory to fractal geometry theory. Besicovitch and Davis prove that for any analytic set, its Hausdorff dimension can be approximated arbitrarily closed by the Hausdorff dimension of its closed subsets. Joyce and Preiss \cite{JP95} prove a similar result for  packing dimension. Recently Slaman proves  that both Besicovitch-Davis and Joyce-Preiss theorem fail for some $\Pi^1_1$-set under the assumption $V=L$. However, we prove that both the theorems hold for arbitrary sets of reals over $\ZF+\sTD$. So the phenomenon  can be viewed as a new regularity property for the sets of reals. After we proved the result, Hirschfeldt and Slaman told us that recently Crone, Fishman and Jackson proved  the following result under stronger assumption. 
\begin{theorem}[Crone,  Fishman and Jackson \cite{crone2020hausdorff}]\label{theorem: cfj theorem}
Assume $\ZF+\AD+\DCR$. For any  set $A$ and $\epsilon>0$, there is a closed set $F\subseteq A$ so that $\DH(F)\geq \DH(A)- \epsilon$.
\end{theorem}

Their proof is direct and uses some rather deep results from set theory. However, we believe our proof is much simpler and more elementary.

\subsection{Point to set principle}

{\em Relativization} opens a door between recursion theory and other mathematical branches. In recursion theory, for a real $x$, a relativization to $x$, roughly speaking, is a way to add prefix $x$- to every appearance  of any notion in the statement.  Then  if a notion is defined in recursion theory,   its relativization is defined {\em naturally}. And if a theorem in recursion theory is proved, then its relativization also follows {\em naturally}. For example, every continuous function is a recursive function relative to a real; and a Borel set is a hyperarithmeitc set relative to a real. From this point of view, one may apply  recursion theory results to analysis.

``Point to set" principle is a more concrete way, by using relativization,  to apply recursion theory to other areas of mathematics. Generally speaking,  the principle says that a set $A$ having certain property is equivalent to that it contains some special points. Such argument can be dated back to Sacks, who (in \cite{Sacks69}) gave a recursion theoretical proof of the classical result that every analytic set is measurable. For one more example, given a relativizable algorithmic randomness notion $\Gamma$ (such as Martin-L\" of-, Schnorr-,   etc), we have the following fact.
\begin{fact}\label{fact: schnorr random }
Assume $\ZF+\CCR$. A set $A\subseteq \mathbb{R}$ is   null if and only if there is some real $x$ so that there is no $\Gamma(x)$-random real in $A$.
\end{fact} 
So if we want to prove that $A$ is not null, then it suffices to prove that for any real $x$, there is a $\Gamma(x)$-random real in $A$. One may also replace randomness with genericity and obtain the the corresponded results. In this paper, we apply some quite recent results in recursion theory and algorithmic randomness theory to descriptive set theory and fractal geometry theory. Especially some deep results concerning the lowness properties for various recursion theory notations turned to be crucial to our proof. The so-called ``lowness properties" is a kind of property  preserving some algorithmic property. For example, a real $x$ is low for Turing jump (or just low) if $x'\equiv_T \emptyset'$; and a real $x$ is called low for Schnorr random (for the definition of Schnorr randomness, see the paragraphs below Theorem \ref{theorem: consequence of TD}) if every Schnorr random real is Schnorr random relative to $x$, etc. Ironically, different than the ``slowdown" properties of themselves,  these notions will be used to prove some ``speedup"  results. We expect to see more such applications in the near future. 

\bigskip

We organize the paper as follows. In Section \ref{section: t and n}, we give some terminologies and notions. In section \ref{section:sami}, we sketch a recursion theoretical reformulation of the Sami's proof that $\ZF+\TD$ implies $\mathrm{CH}$.  The result will be used in Section \ref{section: regular}.  In section \ref{section: wdc}, we prove $\wDC$ within $\ZF+\TD$. In section \ref{section: regular}, we prove that $\ZF+\sTD$  implies regular properties for sets of reals. In the same section, we also prove a basis theorem for  linear orderings over sets of reals within  $\ZF+\TD+\DCR+\mbox{``every uncountable set of reals has a perfect subset"}$. In section \ref{section: capacility}, we prove that Besicovitch-Davis theorem holds for any set of reals within $\ZF+\sTD$.

\section{Terminologies and notions}\label{section: t and n}
 
 We assume that readers have some knowledge of   descriptive set theory and recursion theory. The major references are  \cite{Sacks90}, \cite{CY15book}, \cite{Niesbook09}, \cite{DH10}, \cite{Jech03} and \cite{Ler83}.
\subsection{Set theory}
 
 We assume that readers have some knowledge of axiomatic set theory. $\ZF$ is Zermelo-Fraenkel axiom system. $\AD$ is the axiom of determinacy.
 
 When we say that $T\subseteq 2^{<\omega}$ is a tree, we mean that $T$ is a tree without dead nodes. $[T]$ is the collection of infinite paths trough $T$. Given any $x\in \omega^{\omega}$ and natural number $n$, we use $x\uh n$ to denote an initial segment of $x$ with length $n$.  In other words, $x\uh n$ is a finite string $\sigma\in \omega^{<\omega}$ of length $n$ so that for any $i<n$, $\sigma(i)=x(i)$. 
 
\subsection{Recursion theory}

We use $\leq_T$ to denote Turing reduction and $\leq_h$ to denote hyperarithmetic reduction.  We use $\Phi^x$ denote a Turing machine with oracle $x$. Sometimes we also say that $\Phi^x$ is a recursive functional. We fix an effective enumeration $\{\Phi_e^x\}_{e\in \omega}$ of recursive functionals.

$\KO$ is Kleene's $\KO$. $\CK$ is the least non-recursive ordinal and $\omega_1^x$ is the least ordinal not recursive in $x$.

We say a set $A$ ranges Turing degrees cofinally if for any real $x$, there is some $y\geq_T x$ in $A$. We use $x'$ to denote the Turing jump relative to $x$. More generally, if $\alpha<\omega_1^x$, then $x^{(\alpha)}$ is that $\alpha$-th Turing jump of $x$.

The following fact is folklore and a skeched proof can be found in \cite{PY21}
\begin{lemma}\label{lemma: doubjump}
Assume $\ZF$. For any Turing degree $\x$, there are a family Turing degrees $\{\y_r\mid r\in \mathbb{R}\}$ satisfying the following property:
\begin{itemize}
\item[(1)] For any $r\in \R$, $\x<\y_r$; 
 \item[(2)] For any $r_0\neq r_1\in \R$ and $\z<\y_{r_0},\y_{r_1}$, we have that $\z\leq \x$;
\item[(3)] For any $\z\geq \x''$, the Turing double jump of $\x$, there is an infinite set $C_{\z}\subset \R$ so that $\y_r''=\z$ for any $r\in C_{\z}$.
\end{itemize}
\end{lemma}

\section{On Sami's theorem}\label{section:sami}

\begin{theorem}[Sami \cite{Sami89}]\label{theorem: sami theorem}
$\ZF+\TD+\mathrm{DC}$ proves $\mathrm{CH}$.
\end{theorem}

In this section, we sketch    a recursion theoretical  proof of Theorem \ref{theorem: sami theorem} to show that  $\mathrm{DC}$ can be removed from the assumption, which was also observed by Sami. I.e. we have the following fact.
\begin{proposition}[Sami]\label{proposition: sami theorem}
$\ZF+\TD$ proves $\mathrm{CH}$. 
\end{proposition}
\begin{proof}Given an uncountable  set $A\subseteq \mathbb{R}$. By Lemma \ref{lemma: doubjump}, for any real $x$, there is a real $y>_T x$ so that there is some real $r\in A$ Turing below $y''$ but not below $y$. So, by $\TD$, there is some real $z_0$ so that for any $y\geq_T z_0$, there is some real $r\in A$ Turing below $y''$ but not below $y$.

Now it is simple to construct a $\Sigma^1_1(z_0)$ set $B$\footnote{We  sketch a proof of this and leave the details to readers. First note that the set $\{y\mid \forall r\leq_h z_0(r\leq_T y)\}$ is an uncountable $\Sigma^1_1(z_0)$-set. Then one may construct a perfect set $P\subseteq B$ so that any two different members from $P$ form a minimal pair over $z_0$ in the hyperarithmetic degree sense.}  so that 
\begin{itemize}
\item[(i)] For any $y\leq_h z_0$ and $x\in B$, we have that $y\leq_T x$; and
\item[(ii)] For any $x_0\neq x_1\in B$, if $y\leq_h x_0, x_1$, then $y\leq_h z_0$.
\end{itemize}

Now for any real $x\in B$, we may pick up some real $y_x\in A$ Turing below $x''$ but not below $x$. For any $x_0\neq x_1\in B$, if $y_{x_0}=y_{x_1}$, then by (ii), $y_{x_0}=y_{x_1}\leq_h z_0$. By (i), we have that $y_{x_0}=y_{x_1}\leq_T x_0$, which is a contradiction. 

So $x\mapsto y_x$ is a $1-1$ map from $B$ to $A$. It is known that every uncountable analytic set has  a perfect subset and so $A$ has the same power as $\mathbb{R}$. 
\end{proof}

From the proof of Proposition \ref{proposition: sami theorem}, we can see the following fact that we will use it later in the paper.
\begin{lemma}[Sami \cite{Sami89}]\label{lemma: sami lemma}
Assume $\ZF+\TD$. For any uncountable set $A$ of reals, there is a perfect set $P$ of reals and a sequence of arithmetical  functions \footnote{Actually Sami proves that $f_n$ can be continuous. But we only need this weaker version here.} $\{f_n\}_{n\in \omega}$ from $P$ to $\mathbb{R}$ so that  $P\subseteq \bigcup_{n\in \omega}f_n^{-1}(A)$. Moreover, restricted to $P$, $f_n$ is 1-1 for every $n$.
\end{lemma}
\begin{proof}
Fix an effective enumeration of Turing functional $\{\Phi_n\}_{n\in \omega}$. In the proof of Proposition \ref{proposition: sami theorem}, let $P$ be a perfect subset of $B$. Define $f_n: P\to \mathbb{R}$ so that
\begin{equation}
f_n(x)=\left\{
\begin{array}{rcl}
	\uparrow  & &(\exists m \Phi_n^{x''}(m)\mbox{ is not defined}) \vee( \Phi_n^{x''}\leq_T x);\\
	\Phi_n^{x''} & &  \mathrm{Otherwise}.\\
	\end{array}
	\right
	.
\end{equation}

Clearly $f_n$ is arithmetical for every $n$. Since $P\subseteq B$, we have that $P\subseteq \bigcup_{n\in \omega}f_n^{-1}(A)$. Moreover, if $x\in P$ and $f_n(x)$ is defined, then $f_n(x)\leq_T x''\wedge f_n(x)\not\leq_T x$. Then by the same reason as in the proof of the theorem, $f_n$ must be 1-1 on $P$.  So $\{f_n\}_{n\in \omega}$ is as required.
\end{proof}
\section{Weakly dependent choice}\label{section: wdc}

Throughout the section, we work within $\ZF+\TD$.

\begin{definition}\label{definition: wdc}
Weakly dependent choice for sets of reals, or $\wDC$, says that for any binary relation $R$ over $\mathbb{R}$ with the property that   the set $\{y\mid R(x,y)\}$ has positive inner measure for any real $x$, there is a sequence $\{x_n\}_{n\in \omega}$ of reals so that $\forall n R(x_n,x_{n+1})$.
\end{definition}

\begin{theorem}\label{theorem: consequence of TD}
$\ZF+\TD$ implies $\wDC$

\end{theorem}

We remark that if   ``having positive inner measure"  is replaced with having Baire property and non-meager in the definition of $\wDC$, then the theorem still holds.

A real $r$ is called {\em not Schnorr random} if there is a recursive sequences of recursive open set $\{V_{n}\}_{n\in \omega}$ with $\forall n \mu(V_n)=2^{-n}$ so that $r\in \bigcap_{n\in \omega} V_n$. Otherwise,  $r$ is called {\em Schnorr random}. It is not difficult to see that there is a Schnorr random $r\leq_T \emptyset'$.

A real $x$ is called {\em low for Schnorr random} if every Schnorr random real is Schnorr random relative to $x$. The following theorem, which was proved by Sacks forcing, is due to Terwijn and Zambella.

\begin{theorem}[Terwijn and Zambella \cite{TZ}]\label{theorem: low for schnorr random}
For any real $y\geq_T \emptyset''$, there is a real $x$ low for Schnorr random so that $x''\equiv_T y$.
\end{theorem}
\begin{proof} ({\em  of Theorem \ref{theorem: consequence of TD}.})

Fix any binary relation $R$ as stated in $\wDC$. To prove $\wDC$, we may assume that for any real $x$, the set $R_x=\{y\mid R(x,y)\}$ is upward closed under Turing reduction. I.e. for any $y$ and $z$, if $y\leq_T z$ and $y\in R_x$, then $z\in R_x$. To see this, we may define a new relation   $\tilde{R}$ so that $\tilde{R}(x,y)$ if and only if  for any real $z_0\leq_T x$, there is some real $z_1\leq_T y$ so that $R(z_0,z_1)$. Then for any real $x$, the set $\tilde{R}_x=\{y\mid \tilde{R}(x,y)\}$ is upward closed under Turing reduction and has positive measure, and so co-null. Moreover, if there is a sequence $\{y_n\}_{n\in \omega}$ so that $\forall n\tilde{R}(y_n,y_{n+1})$.  Then we build a sequence $\{x_n\}_{n\in \omega}$ so that $\forall nR(x_n,x_{n+1})$ as follows.

First let $x_0= y_0$. By the definition of $\tilde{R}$, we may pick up the least $m_1$ so that $\Phi_{m_1}^{y_1}$ is defined and $R(x_0, \Phi_{m_1}^{y_1})$. Let $x_1=\Phi_{m_1}^{y_1}$. Generally, if $x_n$ is defined, then $x_n\leq_T y_n$. So by the definition of $\tilde{R}$, we may pick up the least index $m_{n+1}$ so that $\Phi_{m_{n+1}}^{y_{n+1}}$ is defined and $R(x_n, \Phi_{m_{n+1}}^{y_{n+1}})$. Set $x_{n+1}=\Phi_{m_{n+1}}^{y_{n+1}}$. Then we have that $\forall n R(x_n,x_{n+1})$.

Now fix any real $z$, by the assumption on $R$ and Fact \ref{fact: schnorr random }, there is a real $z_0\geq_T z'$ so that for any $y\leq_T z'$ and $z_0$-Schnorr random $r$, $R(y,r)$.  Also by relativizing Theorem \ref{theorem: low for schnorr random} to $z$, there is a real $x>_T z$ low for $z$-Schnorr random so that $x''\geq_T z_0$. So for any $y\leq_T z'$ and $x''$-Schnorr random $r$, $R(y,r)$. Also note that there is a $z$-Schnorr random, and so $x$-Schnorr random, real $r\leq_T z'$. Since $x''\geq_T z'$, there is some index of Turing functional $e$ so that $\Phi_e^{x''}=z'$.  For any number $e\in \omega$, define the set \begin{multline*}A_e=\{x\mid \exists r(r\mbox{ is }x\mbox{-Schnorr random }\wedge r\leq_T \Phi_e^{x''})\\ \wedge \forall r_0\leq_T \Phi_e^{x''}\forall r_1(r_1\mbox{ is }x''\mbox{-Schnorr random}\rightarrow R(r_0,r_1))\}.\end{multline*}

Then by the discussion above, $\bigcup_{e\in \omega}A_e$ ranges Turing degrees cofinally. So there must be some $e_0$ so that $A_{e_0}$ ranges Turing degrees cofinally. By $\TD$, there is some $x_0$ so that for any $y\geq_T x_0$, there is some $y_0\equiv_T y $ in $A_{e_0}$. We may assume that $x_0\in A_{e_0}$. Recursively in $x_0^{(\omega)}$,  we first find a sequence of reals $$\{y_n\in A_{e_0}\mid n<\omega \wedge y_n\equiv_T x_0^{(2n)}\}.$$ Then find a sequence of reals $\{r_n\}_{n\in \omega}$ so that for any $n$,  $r_n\leq_T \Phi_{e_0}^{y_n''}$ is $y_n\equiv_T x^{(2n)}$-Schnorr random. Note for any $n$, $r_n\leq_T \Phi_{e_0}^{y_n''}$ and $r_{n+1}$ is $x^{(2n+2)}\equiv_T y_n''$-Schnorr random (see the figure below). So
by the definition of $A_{e_0}$,  $R(r_n,r_{n+1})$.

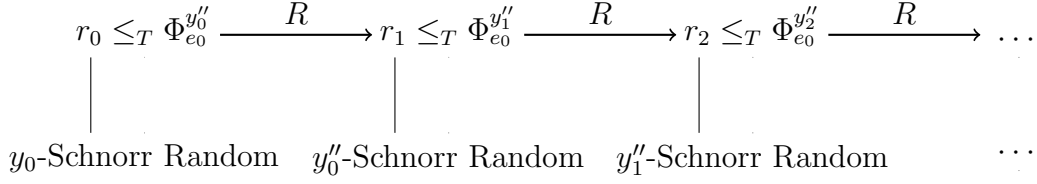
\begin{figure}[h]
\begin{tikzpicture}
\label{figure: figure of de}
  
   \filldraw 
    (3,-2) circle (0pt) node [above,color=black]{$r_0\leq_T \Phi^{y_0''}_{e_0}$}  
    (3,-3) circle (0pt) node [below,color=black]{$y_0$-Schnorr Random}  
    (7,-2) circle (0pt) node [above,color=black]{$r_1\leq_T\Phi^{y_1''}_{e_0}$}  
    (7,-3) circle (0pt) node [below,color=black]{$y_0''$-Schnorr Random}  
    (11,-2) circle (0pt) node [above,color=black]{$r_2\leq_T\Phi^{y_2''}_{e_0}$}  
    (11,-3) circle (0pt) node [below,color=black]{$y_1''$-Schnorr Random}  
                (14.5,-2) circle (0pt) node [above,color=black]{$\cdots$}  
                  (14.5,-3.5) circle (0pt) node [above,color=black]{$\cdots$}  ; 
   
       \draw(2.3,-2) -- (2.3,-3); 
       \draw[->,thick](4,-1.7) -- (6,-1.7);
        \draw(5,-1.7) node [above]{$R$};
   
	       \draw(6.3,-2) -- (6.3,-3); 
       \draw[->,thick](8,-1.7) -- (10,-1.7);
        \draw(9,-1.7) node [above]{$R$};
        	       \draw(10.3,-2) -- (10.3,-3); 
	              \draw[->,thick](12,-1.7) -- (14,-1.7);
        \draw(13,-1.7) node [above]{$R$};

\end{tikzpicture}\caption{$R(r_n,r_{n+1})$}
\end{figure}

 \end{proof}

 \pagebreak

\section{Regularity properties of sets of reals}\label{section: regular}
In this section, we prove some regularity properties for sets of reals under $\ZF+\sTD(+\DCR)$. Woodin already considered  $\sTD$ long time ago. All the results in this section must have been known to him.
\begin{theorem}[Woodin]\label{theorem: conseuqnce of std}
\begin{itemize}
\item[(1)] $\ZF+\sTD$ implies that every set is measurable and has Baire property.
\item[(2)] $\ZF+\sTD$ implies that every uncountable set of reals has a perfect subset.
\end{itemize}
\end{theorem}

\subsection{The proof of part (1)}

We only prove that every set is measurable and leave the second part to readers.

It suffices to prove that for any set $A$, if every measurable subset of $A$ is null, then $A$ must be null. Now suppose that, for a contradiction,  every measurable subset of $A$ is null but $A$ is not null. Then, by Fact \ref{fact: schnorr random } with Schnorr randomness, for any real $z$, there is an $z$-Schnorr random real $z_0$ in $A$. By Theorem \ref{theorem: low for schnorr random} relative to $z$, there is a real $x$ low for $z$-Schnorr random and $x''\geq_T z_0$.

Now for any $e\in \omega$, let $$B_e=\{x\mid \Phi_e^{x''}\in A \mbox{ is an }x\mbox{-Schnorr random real}\}.$$

By the proof above, we have that $\bigcup_{e\in \omega} B_e$ ranges Turing degrees cofinally. Then there is some $e_0$ so that $B_{e_0}$ ranges Turing degrees cofinally. By $\sTD$, there is a pointed subset $P\subseteq B_{e_0}$.

Let $$C=\{r\mid \exists x\in P(\Phi_{e_0}^{x''}=r)\}.$$ $C$ is an analytic set and so measurable. Since $P$ is a pointed set, by the definition of $B_{e_0}$ and Fact  \ref{fact: schnorr random } with Schnorr randomness, $C$ is not null. This is absurd.

\bigskip

\subsection{The proof of part (2)} We first prove the following lemma.
\begin{lemma}\label{lemma 5.2}
Assume $\ZF+\sTD$. For any perfect set $P$ of reals and any partition $P=\bigcup_{n<\omega} B_n$, there exists $n$ such that $B_n$ has a perfect subset.
\end{lemma}
\begin{proof}
Clearly we may assume that $P=2^\omega$ via a homeomorphism. Then for some $n$, $B_n$ ranges Turing degrees cofinally.  By sTD,  $B_n$ contains a perfect subset.
\end{proof}

\begin{proof}(of part (2) of Theorem \ref{theorem: conseuqnce of std}).
Suppose that $A$ is uncountable. By Lemma \ref{lemma: sami lemma}, we may fix  perfect set $P$ and a sequence of functions $\{f_n\}_{n\in \omega}$ as in the Lemma. Then by Lemma \ref{lemma 5.2}, we can choose a perfect $Q\subset f^{-1}_n[A]$ for some $n$. Now  $f_n[Q]$ is an uncountable analytic subset of $A$. So $f_n[Q]$ and hence $A$ contains a perfect subset.
\end{proof}

Here we mention another   approach, within  $\ZF+\TD+\DCR$,  to get a perfect subset due to Sami. A set $A$ of reals is called {\em Bernstein} if  neither $A$ nor $\mathbb{R}\setminus A$ has a perfect subset. Notice that the nonexistence of a Bernstein set implies that for every perfect set $P$ and its subset $A\subseteq P$, either $A$ or $P\setminus A$ has a perfect subset. Sami observed the following relationship between the existence of  a Bernstein set and perfect subset property.
\begin{lemma}[Sami] \label{lemma: no bernstein implies psps}
Assume $\ZF+\TD+\DCR$.  If there is no Bernstein set, then every uncountable set of reals has a perfect subset.
\end{lemma}
\begin{proof}
Suppose that $A$ is uncountable. By Lemma \ref{lemma: sami lemma}, we may fix a perfect set $P$ and a sequence of functions $\{f_n\}_{n\in \omega}$ as in the Lemma.

\bigskip

Let $T^0\subseteq 2^{<\omega}$ be a perfect tree so that $[T^0]=P$. 

Case (1). There is some $\sigma\in T^0$ and some perfect tree $T^0_{\sigma}\subseteq [\sigma]\cap T^0$ so that $f_0$ is defined on every member in $[T^0_{\sigma}]$ and $f_0([T^0_\sigma])\subseteq A$. Fix such $\sigma$ and $T^0_{\sigma}$. Then $f_0([T^0_{\sigma}])\subseteq A$ is an uncountable analytic set. Thus $A$ must have a perfect subset. We are done.

Case (2).  Otherwise. Then by the assumption, for any $\sigma$ with $[\sigma]\cap [T^0]\neq \emptyset$, there is a perfect  tree $[T^0_{\sigma}]\subseteq [\sigma]\cap [T^0]$ so that for any $x\in [T^0_{\sigma}]$, either $f_0(x)$ is not defined or $f_0(x)\not\in A$. Fix some $\sigma_0\in T^0$ so that $[\sigma0]\cap [T^0]\neq \emptyset$ and $[\sigma1]\cap [T^0]\neq \emptyset$. Let $T^0_{\sigma0}$ and $T^0_{\sigma1}$ be as corresponded perfect tree as above. Set $$T^1=\{\tau\in T^0\mid \tau\prec\sigma\vee \tau\succeq \sigma0\vee \tau\succeq \sigma1\}\subseteq T^0.$$ Then for any $x\in [T^1]$,  either $f_0(x)$ is not defined or $f_0(x)\not\in A$.

\bigskip

Process the construction by induction on $n$. Either we stop at Case (1) of some $n$, then we find a perfect subset of $A$. Or else, the construction goes through all of $n$'s. Then by a usual fusion argument, we may find a perfect tree $S$ so that $[S]\subseteq P$ so that for any $x\in [S]$ and any $n$, either $f_n(x)$ is not defined or $f_n(x)\not\in A$. This contradicts to the fact that $[P]\subseteq \bigcup_{n\in \omega} f^{-1}_n(A)$.

Thus we must   stop at Case (1) of some $n$ and so $A$ must have a perfect subset.
\end{proof}

  $\sTD$ implies every set is measurable and so there is no Bernstein set. Thus $\ZF+\sTD+\DCR$ implies every uncountable set of reals has a perfect subset.

\subsection{An application of regular properties to linear orderings over $\mathbb{R}$}

\begin{lemma}\label{lemma: bi null}
Assume $\ZF+\CCR+\mbox{``every sets of reals is measurable"}$. For any   linear order $\leq_L$ over $\mathbb{R}$,  and $A\subseteq \mathbb{R}$ be any non-null set. The collection of the reals $x\in A$ so that either $\{y\in A\mid y\leq_L x\}$ or $\{y\in A\mid x\leq_L y\}$ is null is null.
 \end{lemma}
 \begin{proof}
 Given a  linear order $\leq_L$ over $\mathbb{R}$, Let $A\subseteq \mathbb{R}$ be any non-null set. Fix a non-null set $B\subseteq A$. By   Fubini's theorem, the set $\{(x,y)\mid x\leq_L y\wedge x\in B\wedge y\in B\}$ is measurable and has positive measure.  Let $$L^B=\{x\in B\mid \{y\in A\mid y\leq_L x\} \mbox{ is null}\}$$ be a subset of $B$.  Then by   Fubini's theorem again, the set $B\setminus L^B$ is not null. So the set $$L^A=\{x\in A\mid \{y\in A\mid y\leq_L x\} \mbox{ is null}\}$$ is a null subset of $A$.
 
By the same method, the set $$R^A=\{x\in A\mid \{y\in A\mid x\leq_L y\} \mbox{ is  null}\}$$ is also a null subset of $A$. 
  \end{proof}

 Finally we have the following basis theorem for linear orderings over $\mathbb{R}$ under  $\ZF+\sTD$.
\begin{theorem}\label{theorem: linear order of std}
Assume $\ZF+\DCR+\mbox{``every sets of reals is measurable"}$. For any linear oder $\leq_L$ over $\mathbb{R}$, there is an order preserving embedding from $(2^{\omega},\leq )$ to $(\mathbb{R},\leq_L)$.
\end{theorem}
\begin{proof}
First we set $P_{\emptyset}=[0, 1]$.

By Lemma \ref{lemma: bi null}, there is a  real $x\in P_{\emptyset}$ so that both the sets  $\{y\in A\mid y\leq_L x\}$ and $\{y\in A\mid x\leq_L y\}$ have positive measure. So both of them have disjoint perfect subsets $P_0$ and $P_1$ with positive measure respectively. Moreover, we may require that for any $i\in \{0,1\}$ and $y,z\in P_{i}$, $|y-z|\leq 2^{-1}$.

Now by an induction, it is not difficult to construct a sequence $\{P_{\sigma}\}_{\sigma\in 2^{\omega}}$ of perfect sets so that
\begin{itemize}
\item If $\sigma\succ \tau$, then $P_{\sigma}\subset P_{\tau}$ has positive measure;
\item If $\sigma$ and $\tau$ are incompatible, then $P_{\sigma}\cap P_{\tau}=\emptyset$;
\item If $\sigma$ is in the left of $\tau$, then $\forall x\in P_{\sigma}\forall y\in P_{\tau}(x\leq_L y)$; 
\item For any $\sigma$ and $x,y\in P_{\sigma}$, $|x-y|\leq 2^{-|\sigma|}$.
\end{itemize}

Define $f:2^{\omega}\to \mathbb{R}$ so that $f(x)$ is the unique real in $\bigcap_n P_{x\uh n}$. Then $f$ is an order preserving embedding from $(2^{\omega},\leq)$ to $(\mathbb{R},\leq_L)$.

\end{proof}

One may wonder what happens to Lemma \ref{lemma: bi null} under $\ZF+\TD$. Since it is unknown whether $\ZF+\TD$ implies that every set of reals is measurable, we have to use a more involved argument.

 \begin{definition}
A linear order $(L,\leq_L)$ is {\em locally countable} if for any $l\in L$, the set $\{x\leq_L l\mid x\in L\}$ is countable.
\end{definition}
A typical locally countable order is $(\omega_1,\leq )$.

For any set $A$ of reals which are closed under Turing equivalence relation, a real $x$ is a {\em minimal upper bound} of $A$ if
\begin{itemize}
\item every member of $A$ is recursive in $x$; and
\item there is no real $y<_T x$ so that every member of $A$ is recursive in $y$.
\end{itemize}



 By a classical theorem in recursion theory (see Theorem 4.11 in \cite{Ler83}), for any countable set of reals $A$, there is always a minimal upper bound. 

\begin{lemma}\label{lemma: locally countable} Assume $\ZF+\TD$. There is no uncountable set $A\subseteq \mathbb{R}$ with a locally countable linear order over $A$. 
\end{lemma}


\begin{proof} By Proposition \ref{proposition: sami theorem},  it suffices to prove that there is no locally countable linear order on $\mathbb{R}$. 

 Suppose not. Let $(\mathbb{R},\leq_L)$ be a locally countable order. For any real $x$, let $I_x$ be the Turing downward closure of the set $\{z\mid z\leq_L x\}$. I.e. $$I_x=\{s\mid \exists z\leq_L x(s\leq_T z)\}.$$ Obviously $x\leq_L y$ implies $I_x\subseteq I_y$. 

Note for any real $z$, there is a real $x$ so that $z\in I_x$. So there is a real $z_0\geq_T z$ such that $z_0$ is a minimal upper bound of $I_x$. By $\TD$, there is a real $z_1$ so that every real $z_2\geq_T z_1$ is a minimal upper bound over $I_x$ for some $x$.



 For any real $z$, let $$M_z=\{x\mid     z\mbox{ is a minimal upper bound of }I_x)\}$$ and $$N_z=\bigcup_{x\in M_z}I_x.$$ 

Note that $M_{z_2}$ is nonempty for every $z_2\geq_T z_1$. We have the following fact:
\begin{itemize}

\item For any $z_2, z_3\geq_T z_1$, either $N_{z_3}\subseteq N_{z_2}$ or $N_{z_2}\subseteq N_{z_3}$. Suppose that  $N_{z_3}\not\subseteq N_{z_2}$. Then there must be some $x_3\in M_{z_3}$ so that for any $x_2\in M_{z_2}$, $x_3\not\leq_L x_2$. In other words, $x_2\leq_L x_3$ for any $x_2\in M_2$. So $N_{z_2}\subseteq N_{z_3}$.
\end{itemize}

Now fix a pair of minimal covers $z_2\not\equiv_T z_3$ of $z_1$ (i.e. for $i\in \{2,3\}$, $z_i>_T z_1$ but there is no real $y$ strictly between $z_1$ and $z_i$ in the Turing reduction order sense. For the existence of such a pair, see Lemma \ref{lemma: doubjump}). By the fact above, WLOG, we may assume $N_{z_2}\subseteq N_{z_3}$ and fix some $x\in M_{z_2}$. Then every real in $I_x\subseteq N_{z_2}\subseteq N_{z_3}$ is recursive in both $z_2$ and $z_3$. So every real in  $I_x$ is recursive in $z_1$. Contradicts to the fact that $z_2$ is a minimal upper bound of $I_x$ and $z_1<_T z_2$.
\end{proof}

\begin{corollary}\label{corollary: bi uncountable}
 Assume $\ZF+\TD$. For any uncountable set $A\subseteq \mathbb{R}$ and linear order $\leq_L$ over $A$, there are uncountably many reals $x\in A$ so that both $\{y\in A\mid y\leq_L x\}$ and $\{y\in A\mid x\leq_L y\}$ are uncountable.
 \end{corollary}
 \begin{proof}
 Given a  linear order $\leq_L$ over $\mathbb{R}$. Let $$L=\{x\in A\mid \{y\in A\mid y\leq_L x\} \mbox{ is countable}\}$$ and $$R=\{x\in A\mid \{y\in A\mid x\leq_L y\} \mbox{ is countable}\}.$$
 
 By Lemma \ref{lemma: locally countable}, both $L$ and $R$ are countable. So there must uncountably many reals $x\in A$ so that both $\{y\mid y\leq_L x\}$ and $\{y\mid x\leq_L y\}$ are uncountable.
 \end{proof}
 
 Now we may obtain the following result.
 \begin{theorem}
 Assume $\ZF+\TD+\DCR$. The following are equivalent.
 \begin{enumerate}
 \item  Every uncountable set of reals has a perfect subset. 
 
 \item For any linear oder $\leq_L$ over $\mathbb{R}$, there is an order preserving embedding from $(2^{\omega},\leq )$ to $(\mathbb{R},\leq_L)$.
 \end{enumerate}
 \end{theorem}
 \begin{proof}
 (1)$\Rightarrow$(2). The argument of Theorem \ref{theorem: linear order of std} works here. Just replace ``set with positive measure" by ``uncountable set".\medskip
 
 (2)$\Rightarrow$(1). Fix an uncountable set of reals $A$.  By Proposition \ref{proposition: sami theorem}, $|A|=|\mathbb{R}|$. So $(A, \leq)$ is order isomorphic to $(\mathbb{R}, \leq_L)$ for some $\leq_L$. By (2), there is an order preserving map from $(2^\omega, \leq)$ to $(\mathbb{R}, \leq_L)$ and hence $(A, \leq)$. 
 
 Fix $\pi: 2^\omega\to A$ that preserves order and so is monotonic. Then $\pi$ is continuous on all but countably many points. In particular, $\pi$ is continuous (and injective) on a perfect subset $P$. So $\pi[P]$ is a perfect subset of $A$.
 \end{proof}

\section{Regular property for dimension theory}\label{section: capacility}

For the notions and terminologies in fractal geometry, we follow the book \cite{Fal14}. 

Given a non-empty $U\subseteq \mathbb{R}$, the \emph{diameter} of $U$ is 
$$diam(U)=|U|=\sup\{|x-y| : x, y\in U\}.$$

Given any set $E\subseteq \mathbb{R}$ and $d\geq 0$, let 
$$\mathcal{H}^d(E)=\lim_{\delta\to 0}\inf\{\sum_{i<\omega}|U_i|^{d}: \{U_i\}\mbox{ is an open cover of } E \wedge \forall i\ |U_i|<\delta\},$$ 
\begin{multline*}\mathcal{P}_0^d(E)=\lim_{\delta\to 0}\sup\{\sum_{i<\omega} |B_i|^d: \{B_i\}\mbox{ is a collection of disjoint balls of radii at}\\ \mbox{ most }\delta  \mbox{ with centres in }E\}.\end{multline*}
and  $$\mathcal{P}^d(E)=\inf\{\sum_{i<\omega}\mathcal{P}^d_0(E_i)\mid E\subseteq \bigcup_{i<\omega} E_i\}.$$

\begin{definition}
\begin{itemize} Given any set $E$,
\item[(1)] the {\em Hausdorff dimension} of $E$, or $\DH(E)$, is $$\inf\{d\mid  \mathcal{H}^d(E)=0\};$$
\item[(2)]  the {\em Packing dimension} of $E$, or $\DP(E)$, is $$\inf\{d\mid  \mathcal{P}^d(E)=0\}.$$
\end{itemize}
\end{definition}

By the same reason as in Lebesgue measure, it can be proved with $\ZF+\CCR$ that for any Borel set $B$ and $\epsilon>0$, there is a closed set $F\subseteq B$ so that  $\DH(F)>\DH(B)-\epsilon$.

\begin{theorem}[Besicovitch \cite{Be52} and Davis \cite{Davies52}]\label{theorem: bm theorem}
For any analytic set $A$ and $\epsilon>0$, there is a closed set $F\subseteq A$ so that $\DH(F)\geq \DH(A)- \epsilon$.
\end{theorem}

\begin{theorem}[Joyce and Preiss \cite{JP95}]\label{theorem: jp theorem}
For any analytic set $A$ and $\epsilon>0$, there is a closed set $F\subseteq A$ so that $\DP(F)\geq \DP(A)- \epsilon$.
\end{theorem}

However Slaman proves that both Theorems \ref{theorem: bm theorem} and \ref{theorem: jp theorem} may fail even for some $\Pi^1_1$ set under certain assumptions.
\begin{theorem}[Slaman]
Suppose that the set of constructible reals is not null, then there is a $\Pi^1_1$ set $C$ with $\DH(C)=1$ but for any Borel   $F\subset C$, $\DP(F)=0$.
\end{theorem}

We prove that both Theorems \ref{theorem: bm theorem} and \ref{theorem: jp theorem} remain true for any set of reals under $\ZF+\sTD$.
\begin{theorem}\label{theorem: cfj theorem under zfstd}
$\ZF+\sTD$ implies that for any set of reals $A$ and any $\epsilon>0$, 
\begin{itemize}
\item[(1)] there is a closed set $F\subseteq A$ so that $\DH(F)\geq \DH(A)-\epsilon$.
\item[(2)]there is a closed set $F\subseteq A$ so that $\DP(F)\geq \DP(A)-\epsilon$.
\end{itemize}
\end{theorem}

 To show the theorem,  we use  the ``point-to-set" method. 
 
 Some more facts from algorithmic randomness theory are needed. Let $K$ denote the prefix free Kolmogorov complexity.  We use $K^x$ to denote the prefix free Kolmogorov complexity with oracle, which is a real, $x$. The following  ``point to set" style theorem is  to Lutz and Lutz.
\begin{theorem}[Lutz and Lutz \cite{LL18}]\label{theorem: ll theorem}\footnote{A similar form was also discovered by Cutler. See Theorem 1.4 in \cite{Cutler95}.}  
For any set $A\subseteq \mathbb{R}$, $$\DH(A)=\inf_{x\in\mathbb{R}}  \sup_{y\in A} \underline\lim_{n\to \infty}\frac{K^x(y\uh n)}{n}$$ and $$\DP(A)=\inf_{x\in\mathbb{R}}  \sup_{y\in A} \overline\lim_{n\to \infty}\frac{K^x(y\uh n)}{n}.$$
\end{theorem}

The following lowness property is crucial to our proof.
\begin{theorem}[Herbert \cite{I18}; Lempp, Miller,  Ng, Turetsky,  Weber \cite{Madison14}]\label{theorem: low for hd}
\begin{itemize}
\item There is a perfect tree $T\subseteq 2^{<\omega}$ recursive in $\emptyset'$ so that for any real $x\in [T]$, $$\forall y\in \mathbb{R}(\underline{\lim}_{n\to \infty}\frac{K(y\uh n)}{n}=\underline{\lim}_{n\to \infty}\frac{K^x(y\uh n)}{n}).$$ 
\item  There is a perfect tree $T\subseteq 2^{<\omega}$ recursive in $\emptyset'$ so that for any real $x\in [T]$, $$\forall y\in \mathbb{R}(\overline{\lim}_{n\to \infty}\frac{K(y\uh n)}{n}=\overline{\lim}_{n\to \infty}\frac{K^x(y\uh n)}{n}).$$
\end{itemize}
\end{theorem}

 
 Now we are ready to prove our major theorem of this section.
 \begin{proof}(of Theorem \ref{theorem: cfj theorem under zfstd})
 
 (1). Suppose that $A\subseteq \mathbb{R}$ with $\DH(A)>0$. Fix any $\epsilon>0$. By Theorem \ref{theorem: ll theorem}, for any real $z$, there is some real $x\in A$ so that $$\underline{\lim}_{n\to \infty}\frac{K^z(x\uh n)}{n}>\DH(A)-\frac{\epsilon}{2}.$$ 
 By Theorem \ref{theorem: low for hd} relative to $z$, there is a real $y>_T z$ so that  $$\underline{\lim}_{n\to \infty}\frac{K^y(x\uh n)}{n}=\underline{\lim}_{n\to \infty}\frac{K^z(x\uh n)}{n}>\DH(A)-\frac{\epsilon}{2} \wedge y'>_T x.$$ 
 
 So there must be some $e_0$ so that the set $$B_{e_0}=\{y\mid \Phi_{e_0}^{y'}\in A\wedge \underline{\lim}_{n\to \infty}\frac{K^y(\Phi_{e_0}^{y'}\uh n)}{n}>\DH(A)-\frac{\epsilon}{2}\}$$ ranges Turing degrees cofinally. By $\sTD$, there is a pointed set $P\subseteq B_{e_0}$.
 
 Then the set $$C=\{x\mid \exists y\in P(\Phi_{e_0}^{y'}=x)\}$$ is an analytic subset of $A$. By  Theorem \ref{theorem: ll theorem}, $$\DH(C)>\DH(A)-\frac{\epsilon}{2}.$$ By Theorem \ref{theorem: bm theorem},  $C$ has a closed subset $F$ so that $$\DH(F)>\DH(C)-\frac{\epsilon}{2}.$$ Thus  $$\DH(F)>\DH(A)-\epsilon.$$ 
 
 (2). Same proof as (1) except replacing Hausdorff dimension with packing dimension. We leave the details to readers.
 \end{proof}
 
 To continue our study, we need the following folklore technique lemma of which we sketch a   proof for the completeness.
 \begin{lemma}[Folklore]\label{lemma: bounding}
 Suppose $\ZF+\sTD$. If $f:\mathbb{R}\to Ord$ is a degree invariant (i.e. $x\equiv_T y\implies f(x)=f(y)$ ) map so that $f(x)<\omega_1^x$, then there is an ordinal $\alpha$ so that $f(x)=\alpha$ over an upper cone of Turing degrees.
 \end{lemma}
 \begin{proof}
Fix such a map $f$. Since there are countably many recursive functionals, by $\sTD$, there is some recursive functional $\Phi$ so that $\Phi^x$ codes a linear order for every real $x$; and a pointed set $P$ so that $f(x)\cong \Phi^x$ for any $x\in P$. Let $T$ be a tree representing $P$ so that $\forall x\in P(T\leq_T x)$. Then the set $$\{\Phi^x\mid x\in P\}$$ is a $\Sigma^1_1(T)$ set and so $\Phi^x$ represents an ordinal smaller than $\omega_1^T$ for any $x\in P$ by $\Sigma^1_1$-boundedness relative to $T$ (see \cite{CY15book}). By $\sTD$ again, there must be some $\alpha<\omega_1^T$ and a pointed set $Q\subseteq P$ so that $f(x)=\alpha$ for any $x\in Q$. This finishes the proof.
 \end{proof}
 
 Crone,  Fishman and Jackson also proved the following result.
 \begin{theorem}[Crone,  Fishman and Jackson \cite{crone2020hausdorff}]\label{theorem: cfj theorem}
Assume $\ZF+\AD+\mathrm{DC}$. If $A=\bigcup_{\alpha<\kappa}A_{\alpha}$ for some ordinal $\kappa$, then   $\DH(A)=\sup\{\DH( A_{\alpha})\mid \alpha<\kappa\}$.
\end{theorem}

We  may provide   an ``elementary" proof of the following weaker result under $\ZF+\sTD$.
 
 \begin{theorem}\label{theorem: bounding dimension}
 Assume $\ZF+\sTD$. If $A=\bigcup_{\alpha<\omega_1}A_{\alpha}$, then  $$\DH(A)=\sup\{\DH( A_{\alpha})\mid \alpha<\omega_1\} \mbox{ and }\DP(A)=\sup\{\DP( A_{\alpha})\mid \alpha<\omega_1\}.$$
 \end{theorem}
 \begin{proof}

For any real $x$, let $r=\DH(A)$ and
$$\gamma_x=\min\{\gamma|\sup_{y\in \bigcup_{\alpha< \gamma}A_{\alpha}} \underline\lim_{n\to \infty}\frac{K^x(y\uh n)}{n}\geq r\}.$$ By Theorem \ref{theorem: ll theorem}, $\gamma_x$ is defined for every real $x$.
 
 For any real $z$, by Theorem \ref{theorem: low for hd} and the assumption, there is a real $x>_T z$ so that $\gamma_x=\gamma_z$ but $\omega_1^{x'}>\gamma_z$. So $$\gamma_x=\gamma_z<\omega_1^{x'}=\omega_1^{x}.$$
 
 In other words, $x\mapsto \gamma_x$ is a degree invariant function so that $\gamma_x<\omega_1^x$ over an upper cone of Turing degrees. Then by Lemma \ref{lemma: bounding}, $x\mapsto \gamma_x$ is a constant, say $\eta$,  over an upper cone. Then, by the countability of $\eta$, for any  $m\in \omega$, there must be some $\alpha_m<\eta$ so that the set $\{x\mid \sup_{y\in  A_{\alpha_m}} \underline\lim_{n\to \infty}\frac{K^x(y\uh n)}{n}\geq r-\frac{1}{m}\}$ ranges Turing degrees cofinally. Then  by Theorem \ref{theorem: ll theorem}, $\DH(A_{\alpha_m})\geq r-\frac{1}{m}$.  So $$\DH(A)=\sup\{\DH( A_{\alpha})\mid \alpha<\eta\}=\sup\{\DH( A_{\alpha})\mid \alpha<\omega_1\}.$$ 
 
 We leave the proof of the second part to readers.
 \end{proof}
 
We remark that the conclusion of Theorem \ref{theorem: bounding dimension} can also be proved within $\mathrm{ZFC}+\mathrm{MA}_{\aleph_1}$. Further more,  $\omega_1$ can be replaced with any cardinal $\kappa<2^{\aleph_0}$.
 
\bibliographystyle{plain}

\end{document}